\documentclass[11pt,twoside]{amsart}

\usepackage{latexsym}
\usepackage{amssymb}
\usepackage{vmargin}
\usepackage{amscd}
\usepackage{stmaryrd}
\usepackage{euscript}
\usepackage{mathrsfs}
\usepackage[all]{xy}
\usepackage{xr}
\DeclareMathAlphabet{\mathpzc}{OT1}{pzc}{m}{it}
\usepackage{hyperref}

\hypersetup{colorlinks,
linkcolor=black,
citecolor=black,
urlcolor=blue}

\setmargins{32mm}{20mm}{14.6cm}{22cm}{1cm}{1cm}{1cm}{1cm}

\newcommand\la{\leftarrow}

\newcommand\id{\mathrm{id}}

\newcommand\ten{\otimes}

\newcommand\eps{\epsilon}

\renewcommand\H{\mathrm{H}}
\newcommand\z{\mathrm{Z}}

\newcommand\Z{\mathbb{Z}}
\newcommand\Q{\mathbb{Q}}

\newcommand\R{\mathbb{R}}
\newcommand\Cx{\mathbb{C}}

\newcommand\bG{\mathbb{G}}

\newcommand\bL{\mathbb{L}}

\newcommand\cA{\mathcal{A}}

\newcommand\cC{\mathcal{C}}
\newcommand\cD{\mathcal{D}}

\newcommand\sF{\mathscr{F}}

\newcommand\sO{\mathscr{O}}

\newcommand\fX{\mathfrak{X}}

\renewcommand\L{\Lambda}

\newcommand\g{\mathfrak{g}}

\newcommand\cHom{\mathcal{H}\!\mathit{om}}

\newcommand\CAlg{\mathrm{CAlg}}

\newcommand\Hom{\mathrm{Hom}}

\newcommand\map{\mathrm{map}}

\newcommand\HHom{\underline{\mathrm{Hom}}}

\newcommand\Der{\mathrm{Der}}

\newcommand{\llb}{\llbracket}
\newcommand{\rrb}{\rrbracket}

\newcommand\Spec{\mathrm{Spec}\,}

\newcommand\Set{\mathrm{Set}}

\newcommand\Lim{\varprojlim}
\newcommand\LLim{\varinjlim}

\newcommand\ho{\mathrm{ho}\!}

\newcommand\xra{\xrightarrow}
\newcommand\xla{\xleftarrow}

\newcommand\bt{\bullet}
\newcommand\by{\times}

\newcommand\mc{\mathrm{MC}}

\newcommand\GL{\mathrm{GL}}

\newcommand\gl{\mathfrak{gl}}

\newcommand\et{\acute{\mathrm{e}}\mathrm{t}}

\newcommand\Tot{\mathrm{Tot}\,}
\newcommand\diag{\mathrm{diag}\,}

\newcommand\pd{\partial}

\newcommand\half{\frac{1}{2}}

\newcommand\red{\mathrm{red}}

\newcommand\dR{\mathrm{dR}}
\newcommand\DR{\mathrm{DR}}

\newcommand\op{\mathrm{opp}}

\newcommand\co{\colon\thinspace}

\newcommand\oR{\mathbf{R}}

\newcommand\uleft\underleftarrow
\newcommand\uline\underline
\newcommand\uright\underrightarrow

\newtheorem{theorem}{Theorem}[section]

\newtheorem{corollary}[theorem]{Corollary}

\newtheorem{lemma}[theorem]{Lemma}
\newtheorem*{theorem*}{Theorem}
\newtheorem*{proposition*}{Proposition}
\newtheorem*{corollary*}{Corollary}
\newtheorem*{lemma*}{Lemma}
\newtheorem*{conjecture*}{Conjecture}

\theoremstyle{definition}
\newtheorem{definition}[theorem]{Definition}
\newtheorem*{definition*}{Definition}

\newtheorem*{notation*}{Notation}

\theoremstyle{remark}
\newtheorem{example}[theorem]{Example}
\newtheorem{examples}[theorem]{Examples}
\newtheorem{remark}[theorem]{Remark}

\newtheorem*{example*}{Example}
\newtheorem*{examples*}{Examples}
\newtheorem*{remark*}{Remark}
\newtheorem*{remarks*}{Remarks}
\newtheorem*{exercise*}{Exercise}
\newtheorem*{property*}{Property}
\newtheorem*{properties*}{Properties}

\begin{document}

\begin{abstract}

We explain how any Artin stack $\fX$ over $\Q$ extends to a functor on non-negatively graded commutative cochain algebras, which we think of as functions on Lie algebroids or stacky affine schemes. There is a notion of \'etale morphisms for these CDGAs, and Artin stacks admit \'etale atlases by stacky affines,
giving rise to a small \'etale site of stacky affines over  $\fX$. This site has the same quasi-coherent sheaves as $\fX$ and leads to efficient formulations of shifted Poisson structures, differential operators and deformation quantisations for Artin stacks. There are generalisations to higher and derived stacks.

We also describe analogues for differentiable and analytic stacks; in particular, a Lie groupoid naturally gives  a functor on NQ-manifolds which we can use to transfer structures. In those settings,  local diffeomorphisms and biholomorphisms are the analogues of \'etale morphisms.

This note mostly elaborates  constructions scattered across several of the author's papers, but with an emphasis on the functor of points perspective. New results include consistency checks showing that the induced notions of structures such as vector bundles or torsors on a stacky affine scheme coincide with familiar definitions in terms of flat connections.

\end{abstract}

\title[\'Etale atlases for Artin stacks]{A note on \'etale atlases for Artin stacks and Lie groupoids, Poisson structures and quantisation}

\author{J.P.Pridham}

\maketitle

\section*{Introduction}

Artin stacks were introduced in \cite{Artin} to parametrise moduli problems for which some of the objects being parametrised have positive-dimensional automorphism groups.  Their analogues in differentiable and analytic settings include  Lie groupoids, considered up to Morita equivalence. Artin moduli stacks and their derived versions are a rich source of shifted symplectic structures \cite{PTVV}, important examples being the $(2-d)$-shifted symplectic structures on moduli of local systems on an oriented $d$-manifold and on  moduli of vector bundles on a smooth proper Calabi--Yau variety of dimension $d$.

Where symplectic structures arise,  there should also be Poisson structures and deformation quantisations, and the latter should give rise to deformations of the dg category of perfect complexes (roughly speaking, complexes of vector bundles). However, when attempting to formulate Poisson structures or deformation quantisations for Artin stacks, derived Artin stacks or Lie groupoids, one immediately encounters the difficulty that such structures on affine schemes are only functorial with respect to \'etale morphisms (local diffeomorphisms or biholomorphisms in differentiable and analytic settings). In contrast to the situation for Deligne--Mumford stacks or orbifolds, where local Poisson structures can be glued together, this means that structures on the stack cannot be expected to come from structures on an atlas.

The solution developed in \cite[\S \ref{poisson-stackyCDGAsn}]{poisson}, globalising \cite[Theorem \ref{ddt1-dequiv}]{ddt1}, is to work with
 commutative differential graded algebras (CDGAs) of the form
\[
 B^0 \xra{\pd} B^1 \xra{\pd} B^2 \xra{\pd} \ldots,
\]
which we think of as functions on stacky affine schemes. These should not be confused with the CDGAs arising as functions on derived affine schemes, in which the objects are chain algebras, with differentials mapping the other way. In the differentiable  setting, we impose the additional condition that $B^0$ should be the ring of functions on a $\cC^{\infty}$-space  and that the first differential should be a $\cC^{\infty}$  derivation; NQ manifolds then correspond to finitely generated, locally free objects, and are probably the most studied incarnation.  There are similar generalisations to analytic settings, using Stein spaces and analytic derivations.

In order to compare these CDGAs $B^{\bt}$ with more conventional geometric objects, we exploit the   cosimplicial denormalisation functor $D$ from such CDGAs to cosimplicial commutative rings (Definition \ref{Ddef}). This sends a stacky affine scheme to a simplicial scheme, or an NQ manifold to a simplicial $\cC^{\infty}$-space. Its output is highly singular, with $\Spec D^nB$ being an $n$-nilpotent thickening of $\Spec B^0$. For any sheaf, stack  or simplicial hypersheaf $\fX$ over $\Q$, this  naturally gives rise to a functor $D_*\fX$ on these stacky affine schemes (Definition \ref{Dlowerdef}). The power of this construction lies in the fact that when $\fX$ is an Artin $n$-stack, $D_*\fX$ admits an atlas by stacky affines (Theorem \ref{mainthm}) which is some sense \'etale.  For example, 
an \'etale atlas for $D_*(BG)$ is given by the Chevalley--Eilenberg CDGA of the Lie algebra $\g$, with   its \v Cech nerve taking a similar form (Example \ref{DstarYoverG}).

The consequences of Theorem \ref{mainthm} tend in two main directions. The first direction, exploited extensively in several works listed in the bibliography, is the formulation of concepts such as Poisson structures and deformation quantisations for Artin stacks and Lie groupoids. The theorem means  that in order to define a new structure on an Artin $n$-stack $\fX$, it suffices to define it on the small site of stacky affines which are \'etale over $D_*\fX$.  Since tangent modules are functorial with respect to \'etale morphisms of stacky affines (local diffeomorphisms of NQ-manifolds in the differentiable setting), this allows us efficiently to formulate shifted Poisson structures, differential operators and deformation quantisations for Artin stacks in terms of the small \'etale site of $D_*\fX$ (\S \ref{consequencesn}).

The second direction, which is only implicitly considered in those works, is that it provides an environment where we can compare NQ-manifolds and Artin stacks directly.  Corollary \ref{gooddescent} shows that  $D_*$ has good descent properties for Artin $n$-stacks, providing a full and faithful embedding.  
In the spirit of Grothendieck's functor of points philosophy, we can thus use $D_*$ to think of Artin $n$-stacks or Lie groupoids as functors on stacky affines,  and we can think of any stacky affine, such as an NQ-manifold, as doing the same via the Yoneda embedding. Consequently,  any sufficiently functorial structure on affine schemes with good descent properties has an essentially unique extension to stacky affines; in particular, this applies whenever the moduli functor parametrising the structure is represented by an Artin $n$-stack. For vector bundles and $G$-torsors, the resulting notions for stacky affines take the expected form in terms of flat connections (Examples \ref{BGLnex} and \ref{BGex}). Such structures on an Artin $n$-stack $\fX$ can then be characterised in terms of the corresponding structures on $D_*\fX$; we spell out the case of quasi-coherent sheaves in \S \ref{qcohsn}. 

There are generalisations to derived Artin stacks using stacky derived affines, whose rings of functions are the stacky CDGAs of \cite[\S \ref{poisson-stackyCDGAsn}]{poisson},  given by  introducing a second grading and a chain algebra structure. In particular, for a derived Artin $n$-stack, the associated functor $D_*\fX$ on stacky derived affines admits an \'etale atlas (Theorem \ref{mainthm2}). There are also analogues for (derived) differentiable stacks, in which setting stacky affines correspond to NQ-manifolds, and for (derived) analytic stacks in both Archimedean and non-Archimedean contexts (\S \ref{Fermatsn}).

For simplicity of exposition, we focus on the algebraic case. Readers primarily interested in differentiable or analytic settings might prefer to skip to \S \ref{Fermatsn} first.

\tableofcontents

\section{Stacky affines}\label{stackyaffsn}

\subsection{Basic definitions}

Fix a commutative $\Q$-algebra $R$.

\begin{definition}
 We write $DG^+\CAlg(R)$ for the category of CDGAs over $R$ concentrated in non-negative cochain degrees. Explicitly, an object $A$ of $DG^+\CAlg(R)$ is a cochain complex
\[
 A^0 \xra{\pd} A^1 \xra{\pd} A^2 \xra{\pd} \ldots
\]
of $R$-modules equipped with  associative graded-commutative multiplications $A^m\ten_RA^n \to A^{m+n}$, and a unit $1 \in A^0$, such that
$\pd$ is an   $R$-linear derivation in the graded sense that $\pd(ab)= (\pd a)b+(-1)^{\deg a} a(\pd b)$. 

For $A \in DG^+\CAlg(R)$, we denote by $\Spec A$ the corresponding object of the opposite category  $DG^+\CAlg(R)^{\op}$. 
We also write $A=O(\Spec A)$, and freely make use of the Yoneda embedding to regard $\Spec A$ as the set-valued functor $\Hom(A,-)$ on  $DG^+\CAlg(R)$.
\end{definition}

Beware that in contrast to most instances where cochain complexes crop up, we are primarily interested in these CDGAs up to isomorphism, not up  to quasi-isomorphism.

\begin{example}\label{exYoverg}
The key motivating example 
is given by  a finite projective Lie $R$-algebra $\g$ acting on a commutative $R$-algebra $B$ via a Lie $R$-algebra morphism $\alpha \co \g \to \Der_R(A)$. For $Y=\Spec B$, this leads to the 
CDGA  $O([Y/\g])$
of \cite[Example \ref{poisson-DstarBG}]{poisson} 
given by 
the Chevalley--Eilenberg complex
\[
O([Y/\g]):=(  B \xra{\pd} B\ten \g^{\vee} \xra{\pd} B\ten \Lambda^2\g^{\vee}\xra{\pd} \ldots) 
\]
of $\g$ with coefficients in the  $\g$-module $B$.

The functor  on   $DG^+\CAlg(R)$ associated to  $[Y/\g]:= \Spec O([Y/\g]) $ is then given by the subsets $[Y/\g](B)\subset   Y(B^0) \by (\g \ten_RB^1)$ of pairs $(y,\gamma)$ satisfying the Maurer--Cartan and compatibility conditions
\begin{align*}
 \pd_B\gamma + \half [\gamma,\gamma]&=0   \in \g \ten_RB^2\\ 
y \circ \alpha(\gamma) +\pd_B y  &=0   \in \Der(O(Y),B^1). 
\end{align*}
\end{example}

\begin{example}\label{exOmega}
 The construction of Example \ref{exYoverg} naturally  generalises to Lie algebroids given by Lie--Rinehart algebras. In particular, for a smooth affine scheme $Y$, we can consider the CDGA $\Omega^{\bt}_{Y/R}$ given by the de Rham algebra of $Y$ over $R$. As a functor on $DG^+\CAlg(R)$, 
the associated functor $\Spec \Omega^{\bt}_{Y/R} $ is given by $(\Spec \Omega^{\bt}_{Y/R})(B)\cong Y(B^0)$.
\end{example}

\subsection{\'Etale morphisms}\label{etsn} 

\begin{definition}
 We say that a morphism $f \co C' \to C$ in $DG^+\CAlg(R)$ is a square-zero extension if $f$ is surjective in every level and the kernel $I$ is a square-zero ideal.  We say that $f$ is a contractible square-zero extension if in addition the cochain complex $I$ admits a $C'$-linear (or, equivalently, $C$-linear) contracting homotopy. 
\end{definition}

\begin{definition}\label{fetdef}
Given  functors $F,G$ from $DG^+\CAlg(R)$ to sets, groupoids or simplicial sets, and a natural transformation $\eta \co F \to G$, we say that $\eta$ is formally  \'etale (resp. formally geometric)
if the maps
\[
 F(C') \to F(C)\by^h_{G(C)}G(C')
\]
 are surjective for all square-zero extensions (resp. all contractible square-zero extensions)
$C' \to C$. Here, the symbol $\by^h$ is a homotopy fibre product, which for set-valued functor is is just a plain fibre product, and for groupoid-valued functors is a $2$-fibre product. For groupoids, surjectivity here means essential surjectivity, while for simplicial sets it means $\pi_0$-surjectivity.   

We say that a functor $F$ on $DG^+\CAlg(R)$ is formally  \'etale  (resp. formally geometric)
if the transformation $F \to *$ to the constant functor is so. 
 \end{definition}

\begin{definition}\label{etdef}
We say that  a natural transformation $\eta \co F \to G$ of functors $F,G$ on  $DG^+\CAlg(R)$ is l.f.p. (locally of finite presentation) if for any filtered colimit $C = \LLim_{i \in I} C_i$ in $DG^+\CAlg(R)$, the natural map
\[
 \LLim_{i \in I}F(C_i) \to \LLim_{i \in I}G(C_i)\by_{G(C)}F(C)
\]
is an equivalence. 

We then say that $\eta$ is \'etale (resp. geometric)
if it is  formally  \'etale (resp. formally geometric)
and l.f.p. 

For $A \in DG^+\CAlg(R)$, we refer to $\Spec A$ as a stacky affine if it is geometric.
\end{definition}

Since any CDGA $A \in DG^+\CAlg(R)$ gives rise to a set-valued functor $\Spec A$, the definitions above give rise to notions of (formally) 
\'etale maps between objects of $DG^+\CAlg(R)$. The following lemma is a consequence of  standard obstruction arguments. 
Here, the module  $\Omega^1_{B/A}$ of K\"ahler differentials  is a $B$-module in cochain complexes, spanned by elements $db$ for $b \in B$ subject to the conditions $d(bc)= (-1)^{(\deg b+1)\deg c}  c(db)+(-1)^{\deg b}b(dc)$ and $df(A)=0$. 

\begin{lemma}\label{etlemma}
 A morphism $f \co A \to B$ is geometric if and only if
\begin{itemize}
 \item the graded algebra $B^{\#}$ underlying $B$  (given by ignoring the differential $\pd$)  is freely generated as a graded-commutative algebra over $A^{\#}\ten_{A^0}B^0$ by a finite graded projective $(A^{\#}\ten_{A^0}B^0)$-module.
\end{itemize}

The morphism $f$ is \'etale if and only if it is geometric and
\begin{itemize}
 \item $f^0 \co A^0 \to B^0$ is smooth, and
\item the cochain complex
\[
 \Omega^1_{B/A}\ten_BB^0
\]
 of projective $B^0$-modules is acyclic. 
\end{itemize}
\end{lemma}
%
%

\begin{remark}\label{etrmk}
The final condition in Lemma \ref{etlemma} is the reason for the terminology  ``\'etale'', and will ensure functoriality of constructions such as tangent modules and differential operators with respect to these morphisms.  However, beware that the categorical properties of our \'etale morphisms are more like those for smooth morphisms of affine schemes. In particular, a section of an \'etale map will not tend to be smooth, and if we restrict our functors to the subcategory of $R$-algebras  $\CAlg(R) \subset DG^+\CAlg(R)$, then for an \'etale map
$\eta \co F \to G$, we will only be able to say that $\eta|_{\CAlg(R)} \co F|_{\CAlg(R)} \to G|_{\CAlg(R)}$ is smooth. This is an essential feature, since it  allows smooth maps of affine schemes to have enhancements in stacky affines  which are \'etale in the sense that they preserve cotangent complexes.
\end{remark}

\begin{examples}\label{exYoverg2}
 Given a linear algebraic group $G$ over $R$ with associated Lie algebra $\g$, it follows from the definitions and Example \ref{exYoverg} that  $[G/\g]$ is \'etale over $\Spec R$.

Similarly, for a smooth affine scheme $Y$ over $R$,  the stacky affine  $\Spec \Omega^{\bt}_{Y/R} $ of Example \ref{exOmega}  is \'etale over $\Spec R$.

Since finite projective modules correspond to vector bundles, NQ-manifolds are examples of objects which are geometric over a point, though note that geometricity imposes no smoothness condition in degree $0$.
\end{examples}

\subsection{Denormalisation}\label{denormsn}

To any cosimplicial abelian group $V$,  there is an associated normalised  cochain complex $NV$ given by
\[
 N(V)^n :=\bigcap_{i\ge 0}\ker (\sigma^i: V^n \to V^{n-1})
\]
with differential $\sum_i (-1)^i \pd^i$. By the Dold--Kan correspondence (\cite{W} Theorem 8.4.1, passing to opposite categories and using \cite{W} Lemma 8.3.7), this functor gives an equivalence of categories between cosimplicial abelian groups and cochain complexes concentrated in non-negative degrees. We will be more concerned with the other half of the equivalence, the denormalisation functor $D$ given in level $n$ by the
formal sum 
\[
D^nA:= \bigoplus_{\begin{smallmatrix} m+s=n \\ 1 \le j_1 < \ldots < j_s \le n \end{smallmatrix}} \pd^{j_s}\ldots\pd^{j_1}A^m,
\]
 with the operations $\pd^j$ and $\sigma^i$ defined using the cosimplicial identities
\begin{enumerate}
\item $\pd^j\pd^i=\pd^i\pd^{j-1}\quad i<j$,
\item $\sigma^j\sigma^i=\sigma^i\sigma^{j+1} \quad i \le j$,
\item 
$
\sigma^j\pd^i=\left\{\begin{matrix}
			\pd^i\sigma^{j-1} & i<j \\
			\id		& i=j,\,i=j+1 \\
			\pd^{i-1}\sigma^j & i >j+1
			\end{matrix} \right. ,
$
\end{enumerate}
subject to the conditions that $\sigma^i a =0$ and $\pd^0a= \sum_{i=1}^{n+1}(-1)^i \pd^i a-\pd a $ for all $a \in A^n$. 

\begin{definition}\label{Ddef}
 For any $A \in DG^+\CAlg(R)$, the Dold--Kan denormalisation $DA$ is naturally a cosimplicial commutative $R$-algebra via the Eilenberg--Zilber shuffle product, which reduces to the following description (see for instance  \cite[Definition \ref{ddt1-nabla}]{ddt1}).
Given a finite set  $I$ of strictly positive integers, write $\pd^I= \pd^{i_s}\ldots\pd^{i_1}$, for $I=\{i_1, \ldots i_s\}$, with $1 \le i_1 < \ldots < i_s$. The shuffle product $\nabla$ is then  given on the basis by 
$$
(\pd^Ia)\nabla (\pd^J b):= \left\{ \begin{matrix} \pd^{I\cap J}(-1)^{(J\backslash I, I \backslash J)} (a\cdot b) & |a|= |J\backslash I|, |b|= |I\backslash J|,\\ 0 & \text{ otherwise},\end{matrix} \right.
$$
where for disjoint sets $S,T$ of integers, $(-1)^{(S,T)}$ is the sign of the shuffle permutation of $S \sqcup T $ which sends the first $|S|$ elements to $S$ (in order), and the remaining $|T|$ elements to $T$ (in order). 

Beware that this description  cannot be used to calculate  $ (\pd^Ia)\nabla \pd^J w $ when $0 \in I \cup J$ (for the obvious generalisation of $\pd^I$ to finite sets $I$ of distinct non-negative integers).
\end{definition}

Note that the map $\overbrace{\sigma^0 \circ \sigma^0 \circ \ldots \circ \sigma^0}^n \co D^n A \to A^0$ is surjective, with $n$-nilpotent kernel. Combining these for all $n$ gives a morphism $DA \to A^0$ of cosimplicial algebras (constant cosimplicial structure on the right) which is a nilpotent extension in every level.

\begin{example}
The cosimplicial commutative algebra $DA$ admits a simpler description whenever the CDGA $A$ is freely generated as a graded-commutative $A^0$-algebra by $A^1$, with $A^1$ projective as an $A^0$-module. In particular, this applies whenever $A$ is a  Chevalley--Eilenberg complex of a Lie--Rinehart algebra, as in Examples \ref{exYoverg} and \ref{exOmega}.

In these cases, the commutative algebra $D^nA$ is a commutative  $A^0$-algebra generated by $n$ copies of $A^1$, subject to the relations that for elements  $\alpha^{(r)}, \beta^{(s)}$ in the $r$th and $s$th copies of $A^1$ respectively, we have $\alpha^{(r)}\beta^{(s)}= - \alpha^{(s)}\beta^{(r)}$. Note the the product of any elements in the same copy of $A^1$ is thus zero. 

For $a \in A^0$, $\alpha \in A^1$  and $\pd \alpha = \sum_k \beta_k \wedge \gamma_k \in A^2$,  the coface operations $\pd^i \co D^nA \to D^{n+1}A$ are then given on generators by 
\[
 \pd^i a :=\begin{cases} a & i>0\\ 
            a - (\pd a)^{(1)} & i=0
           \end{cases}
\quad \pd^i \alpha^{(r)}:= \begin{cases} \alpha^{(r)} & i>r \\
                            \alpha^{(r+1)} & 0< i \le r \\
                            \alpha^{(r+1)} - \sum_k\beta_k^{(1)}\gamma_k^{(r+1)} & i=0,
                           \end{cases}
\]
and the codegeneracy operations by
\[
 \sigma^ia :=a, \quad \sigma^i\alpha^{(r)}:= \begin{cases} \alpha^{(r)} & i \ge r \\ \alpha^{(r-1)} & i<r ~\&~ r>1 \\ 0 & i=0 ~\&~ r=1. \end{cases}   
\]

Explicitly, when $A$ is the de Rham CDGA $\Omega^{\bt}_{k[x_1, \ldots, x_m]}$ of a polynomial ring, the ring $D^nA$ can be written as the  commutative $k$-algebra with generators $x_i, \eps_i^{(r)}$ for $1\le i \le m$ and $1 \le r \le n$, and relations $\eps_i^{(r)} \eps_j^{(s)}+ \eps_i^{(s)} \eps_j^{(r)}=0$.
 
\end{example}

We seldom need such explicit descriptions of the denormalisation functor, but we will need a concrete understanding of its left adjoint $D^*$.

\begin{definition}\label{Dstardef}
Given a cosimplicial  commutative algebra $A$, we define the CDGA $D^*A$ as follows. We first consider the  Dold--Kan normalised cochain complex $NA$ given by
\[
 N^mA:= \{a \in A^m ~:~ \sigma^ja = 0 \in A^{m-1}, ~\forall~ 0 \le i <m\},
\]
with differential $\pd a:= \sum(i-1)^i \pd^ia$. We then define an associative (non-commutative) product $\smile$ (a variant of the Alexander--Whitney cup product) on $NA$ by 
\[
a \smile b := (\pd^{[m+1,m+n]}a)\cdot (\pd^{[1,m]}b)  
\]
for  $a \in N^mA$, $b \in N^nA$. 

The commutative cochain algebra  $D^*A$ is then  the quotient of $NA$ by the relations
\[
(\pd^Ia)\cdot (\pd^J b) \sim \left\{ \begin{matrix} (-1)^{(J, I)} (a\smile b) & a \in A^{|J|},~ b\in A^{ |I|},\\ 0 & \text{ otherwise},\end{matrix} \right.
\]
for (possibly empty) sets $I,J$ with $I \cap J= \emptyset$.
\end{definition}

\begin{examples}\label{DstarBG}
Given  a smooth affine group scheme $G$ acting on an affine scheme $Y$, we can form a simplicial scheme $X_{\bt}$ as the nerve of the associated groupoid (with $X_n \cong Y \by G^n$). Functions on this simplicial scheme form a cosimplicial algebra $O(X)$, and then by 
\cite[Example \ref{poisson-DstarBG}]{poisson}, we have $D^*O(X) \cong O([Y/\g])$, the Chevalley--Eilenberg CDGA of Example \ref{exYoverg}.

More generally, if $X_{\bt}$ is the nerve of an affine groupoid scheme with smooth structure maps (or of a Lie groupoid in the differentiable setting), then $D^*O(X)$ is the Chevalley--Eilenberg complex of the associated Lie algebroid.

In particular, if $Y$ is smooth and we let $X_{\bt}$ be the nerve of the groupoid $[Y \by  Y\implies Y]$, then $D^*O(X) \cong \Omega^{\bt}_Y$, the de Rham complex.

\end{examples}

 Thus the  CDGA $D^*O(X)$ is the Chevalley--Eilenberg complex 
\[
O([Y/\g]):=(  O(Y) \xra{\pd} O(Y)\ten \g^{\vee} \xra{\pd} O(Y)\ten \Lambda^2\g^{\vee}\xra{\pd} \ldots) 
\]
of $\g$ with coefficients in the chain $\g$-module $O(Y)$. 

\subsection{Stacks give functors on CDGAs}

The construction $D^*$ from Definition \ref{Dstardef} does not alone suffice for recovering structures on Artin stacks or Lie groupoids from structures on stacky affines, since it clearly loses information and depends on the specific groupoid presentation. In this section we establish an intrinsic relationship to remedy this.

\begin{definition}\label{XKdef}
 Given simplicial sets $K,X$, we denote by $X^K$ the simplicial set given in level $i$ by
\[
 (X^k)_i:= \Hom_{s\Set}(K \by \Delta^i,X).
\]
\end{definition}

\begin{definition}\label{Dlowerdef}
 Given a functor $F$ from $R$-algebras to sets, groupoids or simplicial sets, we define a functor $D_*F$ on  $DG^+\CAlg(R)$ as the homotopy limit
\[
 D_*F(B):= \ho\Lim_{n \in \Delta} F(D^nB).
\]
\end{definition}
For set-valued functors, such as the functor  of points $B \mapsto \Hom(\Spec B, X)$ associated to a scheme $X$, this just gives $D_*F(B)$ as the equaliser of $\pd^0,\pd^1 \co F(B^0) \to F(D^1B)$. 

For groupoid-valued functors, such as the functor of points of a stack, $D_*F(B)$ is equivalent to the groupoid whose objects are pairs $(x,g)$ with $x \in F(B^0)$ and $g\co \pd^0x \to \pd^1x$ an  isomorphism $F(D^1B)$ satisfying the cocycle condition
\[
 \pd^1g = (\pd^2g) \circ (\pd^0g) \co \pd^1\pd^0x \to \pd^2\pd^1x 
\]
in $F(D^2B)$; an isomorphism between $(x,g)$ and $(x',g')$ is then given by an isomorphism $h \co x \to x'$ in $F(B^0)$ satisfying $g' \circ \pd^0h=(\pd^1h) \circ g \co \pd^0x \to \pd^1x'$. 

For simplicial set-valued functors, a model for $D_*F$ is the derived total space
\[
 \oR \Tot F(D^{\bt}B)  =\{ x \in \prod_n \oR F(D^nB)^{\Delta^n}\,:\, \pd^i x_n = \pd_i^{\Delta}x_{n+1},\,\sigma^i x_n = \sigma_i^{\Delta}x_{n-1}\},
\]
of
\cite[\S VIII.1]{sht}, where $\oR F(D^{\bt}B) $ is a  Reedy fibrant replacement of the cosimplicial space $F(D^{\bt}B) $,    and $\pd i^{\Delta}, \sigma_i^{\Delta}$ are defined in terms of the face and degeneracy maps between the simplices $\Delta^n$.

\begin{example}\label{DstarYoverG}
Given a linear algebraic group $G$ acting on an affine scheme $Y$,  the calculation of \cite[Example \ref{poisson-DstarBG}]{poisson} implies that for the associated groupoid $R$-scheme $(Y \by G \Rightarrow Y)$, we have
\[
 D_*(G \by Y \Rightarrow Y)= ([(Y \by G)/(\g \oplus \g)] \Rightarrow  [Y/\g]),
\]
where $[Y/\g]$ is defined in Example \ref{exYoverg}, $\g$ is the Lie algebra of $G$,  and the action of $\g \oplus \g$ on $Y \by G$ is given by the first factor combining the action on $Y$ with the left action on $G$, while the second factor acts via the right action on $G$. One of the maps $ [(Y \by G)/(\g \oplus \g)] \to  [Y/\g]$ is projection onto the first factors, while the other combines the action $Y \by G \to Y$ with projection $\g \oplus \g \to \g$ onto the second factor.
\end{example}

\begin{example}\label{YdRex}
Given a smooth  $R$-scheme $Y$, Simpson's de Rham stack $Y_{\dR}$ from  \cite{simpsonHtpy} is the functor on $R$-algebras defined by $Y_{\dR}(B):= Y(B^{\red})$, for $B^{\red}$ the quotient of $B$ by the ideal of nilpotent elements. Since the natural map $DA \to A^0$ is a nilpotent extension on each level, we get 
\[
 (D_*Y_{\dR})(B) = Y((B^0)^{\red}).
\]
For $Y$ affine,  the natural maps $\H^0B \to B^0 \to (B^0)^{\red}$ thus give natural transformations
\[
 D_*Y \to \Spec \Omega^{\bt}_Y \to D_*Y_{\dR},
\]
for $\Spec \Omega^{\bt}_Y$ as in Example \ref{exOmega}. It should not be surprising that $\Spec \Omega^{\bt}_Y$ differs somewhat from  $D_*Y_{\dR}$, since  $\Omega^{\bt}_Y$  can recover the Hodge filtration while $Y_{\dR}$ cannot.
\end{example}


\begin{definition}\label{sharpdef}
 Given a functor $F$ from  $DG^+\CAlg(R)$ to sets, groupoids or simplicial sets, denote by $F^{\sharp}$ the sheafification (resp. stackification, resp. hypersheafification) of $F$ with respect to those \'etale maps  (cf. Definition \ref{etdef}, Lemma \ref{etlemma}) $A \to B$ in  $DG^+\CAlg(R)$ for which $f^0\co A^0 \to B^0$ is faithfully flat.
\end{definition}
The faithful flatness condition amounts to saying that the associated morphism of schemes is surjective.

\begin{lemma}\label{sharplemma}
 If $F$ is an Artin $n$-stack over $R$, then the natural map $D_*F \to (D_*F)^{\sharp}$ is an equivalence.
\end{lemma}
\begin{proof}
 Take a cosimplicial diagram $B(\bt)$ in $DG^+\CAlg(R)$ equipped with a map $A \to B(\bt)$ from the constant diagram for which the relative latching maps are all \'etale in the sense of Definition \ref{etdef} and faithfully flat in cochain level $0$. We need to show that the map $D_*F(A) \to \ho\Lim_{r \in \Delta}D_*F(B(r))$ is an equivalence.
 
 Since $F$ satisfies smooth hyperdescent and $D_*F(A^0)=F(A^0)$, we know that the map $D_*F(A^0) \to \ho\Lim_{r \in \Delta}D_*F(B(r)^0)$ is a weak equivalence. Now consider the system of square-zero extensions $A^{\le n} \to A^{\le n-1}$ of CDGAs, which induce surjections $D(A^{\le n}) \to D(A^{\le n-1})$ of cosimplicial rings with square-zero kernels $D(A^n[-n])$. Since $F$ has an obstruction theory in the sense of \cite[Definition 1.4.2.2]{hag2}, taking homotopy fibres over $x \in F(A^0)$ gives  homotopy fibre sequences
 \begin{align*}
   F(D^m(A^{\le n}))_x \to  F(D^m(A^{\le n-1}))_x &\to \oR\HHom_{A^0}(x^*\bL_{F/R},D^m(A^n[-n])[1]) \\
   D_*F(A^{\le n})_x \to  D_*F(A^{\le n-1})_x &\to \oR\HHom_{A^0}(x^*\bL_{F/R},A^n[1-n]),
 \end{align*}
and similarly for each $B(r)$, where the first line implies the second  
by taking  $\ho\Lim_{m \in \Delta}$. 

Since the relative latching maps of $B(\bt)$ over $A$ are all faithfully flat as morphisms of graded algebras, they satisfy cohomological descent, so $\oR\HHom_{A^0}(x^*\bL_{F/R},A^n) \simeq \ho\Lim_{r \in \Delta} \oR\HHom_{A^0}(x^*\bL_{F/R},B(r)^n)$ for all $n$. We thus inductively deduce that the maps $D_*F(A^{\le n}) \to \ho\Lim_{r \in \Delta}D_*F(B(r)^{\le n})$ are all  weak equivalences, and the conclusion follows by passing to the homotopy limit over all $n$, noting that $D^m(B^{\le n}) \cong D^mB$ for all $m \le n$.
\end{proof}
\begin{remark}
The proof of Lemma \ref{sharplemma} applies to more general functors: the only properties required for $F$ are that it satisfies descent with respect to smooth hypercovers and that it has a cotangent complex giving an obstruction theory.
\end{remark}

\begin{example}\label{BGLnex}
 Associated to the algebraic group $\GL_n$, there is a stack $B\GL_n$ parametrising rank $n$ vector bundles, and we may combine Definition \ref{Dlowerdef} 
 and Lemma \ref{sharplemma}
 to give a stack $(D_*B\GL_n)$ on $DG^+\CAlg(R)$ with respect to the  stacky affine \'etale covers. Using Example \ref{DstarYoverG}, we may characterise this as the stackification of the groupoid functor $([\GL_n/(\gl_n \oplus \gl_n)] \Rightarrow  [\ast/\gl_n])$. 

It follows that $(D_*B\GL_n)(B)$ is equivalent to the  groupoid  of pairs $(M,\nabla)$ for $M$ a projective $B^0$-module (equivalently, a vector bundle) of rank $n$ and $\nabla \co M \to B^1\ten_{B^0}M$ a flat $\pd_B$-connection, i.e. a degree $1$ operator $\nabla$ on $B^{\#}\ten_{B^0}M$ satisfying $\nabla \circ \nabla =0$ and $\nabla(bm)= \pd(b)m+(-1)^{\deg b}\nabla(m)$ for $b \in B^{\#}, m \in M$. 

Explicitly, we can see this by first unwinding the definitions,  characterising elements of $[\ast/\g](B)$ as Maurer--Cartan elements 
\[
 \{\omega \in \g\ten B^1 ~:~ \pd_B\omega + \half[\omega,\omega]=0 \in \g \ten B^2\} =: \mc(\g\ten B),
\]
and elements of $[G/(\g \oplus \g)](B)$ as triples $(g, \omega_0, \omega_1) \in G(B^0) \by \mc(\g \ten B)^2$ such that $\omega_1= g\omega_0g^{-1} +dg.g^{-1}$, meaning that $g$ is a gauge transformation.

The source and target maps $\pd_0,\pd_1 \co [G/(\g \oplus \g)](B) \to [\ast/\g](B)$ are given by $\pd_i(g, \omega_0, \omega_1)= \omega_i$, while the product $[G/(\g \oplus \g)](B)\by_{ [\ast/\g](B)} [G/(\g \oplus \g)](B) $ sends $((g, \omega_0, \omega_1), (h,\omega_1, \omega_2))$ to $(hg,\omega_0, \omega_2)$.

Each object $\omega \in [\ast/\gl_n](B)$ thus corresponds to a flat connection $\nabla_{\omega}:=\pd_B + \omega$ on $(B^0)^{\oplus n}$, while each morphism $(g, \omega_0, \omega_1)\in [\GL_n/(\gl_n \oplus \gl_n)](B)$ corresponds to the element $g \in \GL_n(B^0)$ acting as an automorphism of  $(B^0)^{\oplus n}$ intertwining $\nabla_{\omega_0},\nabla_{\omega_1}$. 

The groupoid $([\GL_n/(\gl_n \oplus \gl_n)] \Rightarrow  [\ast/\gl_n])$ in presheaves thus represents the groupoid of  flat $\pd_B$-connections $\nabla$ on free rank $n$ $B^0$-modules $M$.  On applying \'etale descent, we recover the groupoid of pairs $(M,\nabla)$ with $M$ locally free, since every locally free module is free on some \'etale cover. 
%
\end{example}

\begin{example}\label{BGex}
 For an arbitrary linear algebraic group $G$, we can consider the stack $BG$ parametrising $G$-torsors.  A similar argument to the previous example characterises $(D_*BG)(B)$
as the groupoid of pairs $(P,\nabla)$, for $P$ a $G$-torsor over $B^0$ and $\nabla \co O(P)\to O(P)\ten_{B^0}B^1$ a flat connection acting as a differential with respect to the multiplicative structure on the ring of functions $O(P)$, required to be $G$-equivariant in the sense that $\mu \circ \nabla = (\nabla \ten \id) \circ \mu$, for $\mu \co O(P) \to O(P)\ten O(G)$ the co-action.
\end{example}

\section{\'Etale atlases and quasi-coherent sheaves}

\subsection{\'Etale atlases}

Given a sheaf or stack  $F$ on $DG^+\CAlg(R)$ with respect to the \'etale covers of Definition \ref{sharpdef}, we can define \'etale atlases or $1$-atlases exactly as we do for algebraic spaces or Deligne--Mumford stacks \cite{champs}, in terms of \'etale covers $U \to F$ satisfying relative representability conditions. Likewise for simplicial hypersheaves on $DG^+\CAlg(R)$, we have a notion of $(n, \et)$-geometricity by following the yoga of   \cite{simpsonAlgGeomNstacks,hag2} (the terminology of \cite{lurie} is slightly different --- see \cite[Remark \ref{stacks2-cflurie}]{stacks2} for a comparison). This leads to the following key theorem:

\begin{theorem}\label{mainthm}
 Given an  Artin stack $\fX$ over $R$, the stack $D_*\fX$ on $DG^+\CAlg(R)$ admits an \'etale atlas by a cover of stacky affines. More generally, for any $n$-geometric  Artin stack $\fX$ over $R$, the hypersheaf $D_*\fX$ is $(n, \et)$-geometric.
\end{theorem}
\begin{proof}[Proof (sketch).]
We follow the construction of \cite[\S \ref{poisson-stackyCDGAsn}]{poisson}. The essential step is given by \cite[Theorem \ref{stacks2-relstrict}]{stacks2} (see also \cite[Theorem \ref{stacksintro-strict}]{stacksintro}), which provides an Artin $(n+1)$-hypergroupoid $X_{\bt}$ resolving  $\fX$ ---  this is a simplicial scheme  with all the partial matching maps (and hence all the face maps $\pd_i$) being smooth surjections, and each $X_n$ being a disjoint union of affine schemes. When $\fX$ has affine diagonal (corresponding to $0$-geometric in the later versions of \cite{hag2}), then this resolution is straightforward to construct as a \v Cech nerve, but beware that the higher cases are fairly non-trivial: if $\fX$ is $n$-geometric (so  $\fX \to \fX^{h S^n}$ is affine), then the construction has $2^{n+1}-1$ steps.

We first address the case where $\fX$ is strongly quasi-compact in the sense of \cite{hag2}, meaning that all the higher diagonals $\fX \to \fX^{h S^n}$ are quasi-compact for $n \ge -1$ (where $S^{-1}=\emptyset$). Then we may assume that each scheme $X_n$ is affine, giving us a cosimplicial $R$-algebra $O(X)$. The denormalisation functor $D$ has a left adjoint $D^*$ as in Definition \ref{Dstardef}, and unwinding Definition \ref{Dlowerdef} it follows that $(D_*\fX)$ is equivalent to the hypersheafification of the simplicial object $\Spec D^*O(X^{\Delta^{\bt}})$ given by 
\[
 n \mapsto \Spec D^*O(X^{\Delta^n}),
\]
 for $X^{\Delta^n}$ as in Definition \ref{XKdef}. 

Looking at the effect of the functor $D$ on square-zero extensions and on contractible square-zero extensions, together with the Artin hypergroupoid conditions from \cite[Definition \ref{stacks2-nptreldef}]{stacks2} or \cite[Definition \ref{stacksintro-npdef}]{stacksintro} on the resolution $X$, it follows that the partial matching maps of $\Spec D^*O(X^{\Delta^{\bt}})$  are \'etale while the matching maps are geometric, in the sense of Definition \ref{etdef}. In particular, this means that each  $\Spec D^*O(X^{\Delta^{n}})$ is a stacky affine, and in the terminology of \cite{stacks2}, the simplicial object $\Spec D^*O(X^{\Delta^{\bt}})$ is an $(n+1,\et)$-hypergroupoid in stacky affines via \cite[Lemma \ref{stacks2-powers}]{stacks2}. By \cite[Proposition \ref{stacks2-easy}]{stacks2}, this implies that $(D_*\fX)$ is $(n, \et)$-geometric, with $\Spec D^*O(X) \to \fX$ being an \'etale $n$-atlas.

If $\fX$ is not strongly quasi-compact, we take the simplicial diagram $X_{\bt}$ of disjoint unions of affine schemes, and write $X_0= \coprod_{\alpha} U_{\alpha}$. For $B \in DG^+\CAlg(R)$, a morphism $\Spec DB \to X_{\bt}$ of simplicial schemes always factors through the completion of $X_{\bt}$ along the iterated degeneracy map  $\uline{\sigma} \co X_0\to X_{\bt}$, because the natural map $DB \to B^0$ is a nilpotent extension on each level. In particular, each such morphism factors through the localisation at $X_0$, and the constructions above adapt if we replace each instance of $\Spec D^*O(X)$ with the coproduct $\coprod_{\alpha} \Spec D^*\Gamma(U_{\alpha}, \sigma_{\bt}^{-1}\sO_X)$, with a similar replacement for each $\Spec D^*O(X^{\Delta^n})$ based on affine covers of each $X_n$.
\end{proof}

\begin{example}\label{atlasYoverG}
For a linear algebraic group $G$ acting on an affine scheme $Y$, we can apply the theorem to the quotient stack $[Y/G]$, in which case we just recover the description of Example \ref{DstarYoverG}. Explicitly,
 an \'etale atlas for $D_*[Y/G]$ is given by the stacky affine $[Y/\g]$, and the associated simplicial resolution is the nerve of the groupoid $([(Y \by G)/(\g \oplus \g)] \Rightarrow  [Y/\g])$ of stacky affines.
\end{example}

In the following Corollary, note that if the target  $F$ is also an Artin $n$-stack then the hypersheafification is redundant by Lemma \ref{sharplemma}:
\begin{corollary}\label{gooddescent}
For any simplicial \'etale hypersheaf  $F$ on affine $R$-schemes and any Artin $n$-stack $\fX$, the canonical natural transformation 
\[
 \map(\fX,F) \to  \map(D_*\fX  , (D_*F)^{\sharp})
\]
of simplicial mapping spaces is a weak equivalence. This space is moreover equivalent to the homotopy limit of $D_*F$ evaluated on the simplicial category of stacky affines \'etale over $D_*\fX$. 
\end{corollary}
\begin{proof}
 As in the proof of  Theorem \ref{mainthm}, $\fX$ admits an Artin $(n+1)$-hypergroupoid resolution $X_{\bt}$, and  $D_*\fX$ is equivalent to the hypersheafification of a  simplicial object $\Spec D^*O(X^{\Delta^{\bt}})$ (with a slight variation in the non-strongly quasi-compact case). As in \cite[Corollary \ref{poisson-gooddescent}]{poisson}, both of our assertions then reduce to the statement that the map
\[
X \to  \diag \Spec DD^*O(X^{\Delta^{\bt}})
\]
is a weak equivalence of simplicial presheaves. 

This last statement follows from \cite[Proposition \ref{poisson-replaceprop}]{poisson} (or rather its proof, the fibrant hypotheses not being strictly necessary). The key idea is that $D$ and $D^*$ descend to functors between graded-commutative algebras and almost cosimplicial commutative algebras, giving rise for each $j$ to  contracting homotopies
\[
 D^j(D^*O((X^{\Delta^{\bt}})_i)^{\#})\to  D^j(D^*O((X^{\Delta^{\bt}})_i)^{\#})^{\Delta^1_i}
\]
making the natural map  $  D^j(D^*O((X^{\Delta^{\bt}})_i)^{\#})\to O(X_i)$ a deformation retract, functorial in $i$ with respect to the simplicial operations.
\end{proof}

\subsection{Quasi-coherent sheaves and torsors}\label{qcohsn}

We can apply Corollary \ref{gooddescent} to stacks $F$ such as $B\GL_n$, $BG$, the stack of quasi-coherent sheaves or the higher stack of perfect complexes. We now give the resulting  characterisations of the groupoids of vector bundles, of $G$-torsors, of quasi-coherent sheaves and of perfect complexes on $\fX$, all in terms of data on the site of stacky affines \'etale over  $D_*\fX$. 

\begin{example}[Vector bundles]\label{BGLnex2}
Following Example \ref{BGLnex}, 
if we define a  vector bundle on a stacky affine $\Spec B$ to be a pair $(M,\nabla)$ for $M$ a projective $B^0$-module of rank $n$ and $\nabla \co M \to B^1\ten_{B^0}M$ a flat $\pd_B$-connection, then it follows from Corollary \ref{gooddescent} applied to $\map(\fX,B\GL_n)$  that the groupoid of vector bundles on an Artin stack $\fX$ is equivalent to the groupoid of vector bundles on the site of stacky affines \'etale over $D_*\fX$.
\end{example}

\begin{example}[Torsors]
Following   Example \ref{BGex},
if for an arbitrary linear algebraic group $G$ we define a $G$-torsor  on a stacky affine $\Spec B$ to be a
 pair $(P,\nabla)$ for $P$ a $G$-torsor over $B^0$ and $\nabla \co O(P)\to O(P)\ten_{B^0}B^1$ a $G$-equivariant flat connection, then it follows from Corollary \ref{gooddescent}  applied to $\map(\fX,BG)$ that the groupoid of $G$-torsors on an Artin stack $\fX$ is equivalent to the groupoid of $G$-torsors on the site of stacky affines \'etale over $D_*\fX$.
 \end{example}

\begin{definition}\label{Cartdef}
Given a  stacky affine $\Spec B$, define a Cartesian $B$-module to be a $B$-module $M$ in cochain complexes for which the natural map
\[
 M^0\ten_{B^0}B^{\#} \to M^{\#} 
\]
of graded $B^{\#}$-modules is an isomorphism.
\end{definition}
Equivalently, we could characterise a  Cartesian $B$-module as a pair $(M^0,\nabla)$, for a $B^0$-module $M$ and a flat $\pd_B$-connection $\nabla \co M^0 \to M^0\ten_{B^0}B^1$. Given $M$, the associated pair is $(M^0, \pd_M\co M^0 \to M^1)$.

\begin{example}[Quasi-coherent sheaves]\label{QCohex}
Applying Corollary \ref{gooddescent} to the stack of quasi-coherent sheaves, the reasoning of Example \ref{BGLnex2} adapts to show that the groupoid of  quasi-coherent sheaves  on an Artin stack $\fX$ is equivalent to the groupoid of Cartesian modules on the site of stacky affines \'etale over $D_*\fX$. Looking at stacks of morphisms of  quasi-coherent sheaves, we can then conclude the the category of  quasi-coherent sheaves  on an Artin stack $\fX$ is equivalent to the category of quasi-coherent Cartesian $\sO$-modules on the site of stacky affines \'etale over $D_*\fX$.
 \end{example}

\begin{definition}\label{quCartdef}
Given a  stacky affine $\Spec B$, define a quasi-Cartesian $B$-complex to be a $B$-module $M=\bigoplus_{i,n} M^n_i$ in chain  cochain complexes for which the natural map
\[
 M^0\ten_{B^0}B^{\#} \to M^{\#} 
\]
of graded $B^{\#}$-modules is an quasi-isomorphism of graded chain complexes.
\end{definition}

\begin{example}[Quasi-coherent complexes]\label{perex}
Following \cite[\S 5.4.2]{stacks2},  we can define  a quasi-coherent complex on an Artin stack $\fX$ to be a presheaf $\sF_{\bt}$ of $\sO_{\fX}$-modules whose homology \emph{presheaves} $\H_i(\sF_{\bt})$ are all quasi-coherent. Then Example \ref{QCohex} generalises to the statement that the $\infty$-category of  bounded below quasi-coherent chain complexes on $\fX$ is equivalent to the category of homologically bounded below quasi-coherent quasi-Cartesian $\sO$-modules on the site of stacky affines \'etale over $D_*\fX$, localised at levelwise quasi-isomorphisms. 

A similar statement holds for perfect complexes; as observed in  \cite[Proposition \ref{DQnonneg-Perprop2}]{DQnonneg}, it follows by combining \cite[Proposition \ref{stacks2-qcohequiv}]{stacks2} with \cite[Lemma \ref{poisson-denormmod} and Corollary \ref{poisson-gooddescent}]{poisson}. 
\end{example}

\section{Tangent modules, Poisson structures, differential operators 
and quantisations}\label{consequencesn}

\subsection{Cotangent and tangent complexes}

\subsubsection{Cotangent complexes and symplectic structures}

On a smooth Artin $n$-stack $\fX$, the cotangent complex $\bL_{\fX}$ is a perfect complex of amplitude $n$. 

From the constructions of \cite[\S \ref{stacks2-cotsn}]{stacks2} and \cite[Proposition \ref{poisson-tgtcor2}]{poisson}, it follows that the  complex of derived global sections $\oR\Gamma(\fX, \L^p\bL_{\fX})$ is quasi-isomorphic to the complex of derived global sections of the presheaf $B \mapsto \Omega^p_B$ on the  site $(D_*\fX)_{\et}$  of stacky affines $\Spec B$ \'etale over $D_*\fX$.

We can also apply this to the de Rham complex and its filtered pieces, so as in \cite[Lemma 3.25]{poisson}, 
\[
 F^p\DR(\fX) \simeq \oR\Gamma((D_*\fX)_{\et}, F^p\Omega^{\bt}_{\sO}),
\]
while the space $\cA^{p,cl}_R(\fX,n)$ of closed $p$-forms of degree $n$ from \cite{PTVV} is just the Dold-Kan denormalisation of the good truncation in non-positive degrees of 
\[
 \oR\Gamma((D_*\fX)_{\et}, F^p\Omega^{\bt}_{\sO})[n+p].
\]
In particular, for $p=2$, the space of $n$-shifted symplectic structures \cite{PTVV} 
then consists  of the non-degenerate elements in this space (for $\Z/2$-graded manifolds, these were previously introduced in  \cite{KhudaverdianVoronov} as homotopy symplectic structures).  

\subsubsection{Tangent complexes}\label{tgtsn}

On a stacky affine $\Spec B$, we can use the internal $\Hom$-functor  $\cHom_B$ for $B$-modules in complexes to give a complex $\cHom_B(\Omega^1_B,B)$ of $B$-modules. When $B^0$ is smooth, and $B\to C$ is \'etale in the sense of  Definition \ref{etdef}, note that the canonical map
\[
 \cHom_B(\Omega^1_B,B)\ten_BC \to \cHom_C(\Omega^1_C,C)
\]
is a quasi-isomorphism. By \cite[Proposition \ref{poisson-tgtcor2}]{poisson}, we then have a quasi-isomorphism 
\[
 \oR\HHom_{\sO_{\fX}}(\bL_{\fX},\sO_{\fX}) \simeq  \oR\Gamma((D_*\fX)_{\et}, \cHom_{\sO}(\Omega^1_{\sO},\sO)),
\]
and indeed 
\[
 \oR\HHom_{\sO_{\fX}}(\bL_{\fX}^{\ten p},\sO_{\fX}) \simeq  \oR\Gamma((D_*\fX)_{\et}, \cHom_{\sO}((\Omega^1_{\sO})^{\ten p},\sO)).
\]

\subsection{Poisson structures and quantisations}\label{poissonquantsn}

The expressions in \S \ref{tgtsn} 
 lead to well-behaved complexes of polyvectors and shifted polyvectors on $\fX$, which in particular carry Schouten--Nijenhuis  brackets, permitting the formulation of shifted Poisson structures \cite[\S \ref{poisson-bipoisssn}]{poisson}. In brief, an $n$-shifted Poisson structure on a stacky affine $\Spec B$ is an enrichment of the CDGA structure on $B$ to a strong homotopy $P_{n+1}$-algebra structure, and an $n$-shifted Poisson structure on $\fX$ is an  $\infty$-functorial choice of $n$-shifted Poisson structure for each stacky affine \'etale over $D_*\fX$. 
 
 The equivalence between $n$-shifted symplectic structures and non-degenerate $n$-shifted Poisson structures in this sense is then given by \cite[Theorem \ref{poisson-Artinthm}]{poisson}, by solving a tower of obstruction problems; a non-homotopical forerunner of that construction appears for graded manifolds in \cite{KhudaverdianVoronov}, interpreted as a Legendre transformation. Over a Noetherian base ring, shifted Poisson structures in our sense are expected to coincide with the more involved formulation  of \cite{CPTVV}; the correspondence between symplectic and non-degenerate Poisson structures is also established  in \cite{CPTVV}, using a less direct method.

Hochschild complexes and rings of differential operators admit canonical filtrations, whose associated gradeds are complexes of (shifted) polyvectors. They also therefore satisfy functoriality with respect to the \'etale morphisms of Definition \ref{etdef}, and deformation quantisations can be formulated in terms of almost commutative deformations of the structure sheaf $\sO$ on $(D_*\fX)_{\et}$. Via Example \ref{perex}, $0$-shifted quantisations then lead to deformations of the dg category of perfect complexes on $\fX$.

As observed in \cite{CPTVV}, for positively shifted structures the  existence of quantisations follows immediately from formality of the $E_{n+1}$-operads, leading to deformations of perfect complexes as an $n$-tuply monoidal dg category. For $1$-shifted Poisson structures on classifying stacks $BG$, deformation quantisations are then related to quantum groups. Our smoothness hypotheses above do not lead to many interesting negatively shifted structures; in particular, $n$-shifted symplectic structures on smooth Artin $N$-stacks only exist for $n\ge 0$. There are also  many quantisation results for negatively shifted structures on singular stacks when regarded as derived stacks; see \S \ref{derivedquantsn}.

A shifted co-isotropic structure on a morphism  $\Spec B \to \Spec A$  of stacky affines can be formulated in terms of an $n$-shifted Poisson structure on $A$, an $n-1$-shifted Poisson structure on $B$ and a homotopy $P_{n+1}$-algebra morphism from $A$ to twisted shifted polyvectors on $B$  (cf. \cite{CPTVV,MelaniSafronovI}). For $n\ge 2$, quantisations of such structures are established in \cite{MelaniSafronovII}  and \cite{DQpoisson} covers $n=1$, while  for $n=0$, quantisations are established in non-degenerate cases (i.e. $0$-shifted Lagrangians) in \cite{DQLag}.

\section{Generalisations: derived and analytic structures}

\subsection{Derived Artin stacks}\label{derivedsn}

In order to simplify the exposition, we have restricted our attention so far to higher stacks and assumed smoothness, avoiding any derived structure. The various results cited above are all formulated in the more general setting of derived stacks, and we now briefly explain what extra structure has to be introduced and where the subtleties lie.

\subsubsection{Stacky CDGAs}

We will be working systematically with chain cochain complexes $V$,  which are bigraded abelian groups $V= \bigoplus_{i,j}V^i_j$, equipped with square-zero linear maps $\pd \co V^i_j \to V^{i+1}_j$ and $\delta \co V^i_j \to V^i_{j-1}$ such that $\pd\delta + \delta \pd =0$.

The following is taken from \cite[Definition 3.2]{poisson}:
\begin{definition}
We define a  stacky CDGA  to be a chain cochain complex $A$ of $\Q$-vector spaces equipped with a commutative product $A\ten A \to A$ and unit $\Q \to A$. 

We  regard all chain complexes as   chain cochain  complexes $V= V^0_{\bt}$.
Given a chain  CDGA $R$, a stacky CDGA over $R$ is then a morphism $R \to A$ of stacky CDGAs. We write $DG^+dg_+\CAlg(R)$ for the category consisting of stacky CDGAs $A$ over $R$ concentrated in non-negative bidegrees (i.e.  $A^i_j=0$ unless $i,j \ge 0$).

For $A \in DG^+dg_+\CAlg(R)$, we denote by $\Spec A$ the corresponding object of the opposite category  $DG^+dg_+\CAlg(R)^{\op}$; 
we also write $A=O(\Spec A)$.
\end{definition}

There are obvious analogues of Examples \ref{exYoverg} and \ref{exOmega}, giving stacky CDGAs $O([Y/\g])$ and  $\Omega^{\bt}_{Y/R}$ for non-negatively graded chain CDGAs $O(Y)$; see \cite[Example \ref{poisson-DstarBG}]{poisson} for details of $[Y/\g]$. These behave well when 
$O(Y)_0$ is smooth and $O(Y)_{\#}$ is freely generated over $O(Y)_0$ by a graded projective module.

\begin{example}\label{Hamex}
 A source of motivating examples of such structures is given by an infinitesimal version of the derived Hamiltonian reduction $[\oR\mu^{-1}(0)/G]$ of \cite{calaqueLagrangian}. Take a smooth affine scheme $Y$ equipped with the action of an algebraic group $G$, together with a $G$-equivariant morphism $\mu \co Y \to \g^*$. We can then form the derived vanishing locus $\oR \mu^{-1}(0)$ of $Y$, whose functions form the chain CDGA 
\[
 O(\oR\mu^{-1}(0)):= (O(Y) \xla{\delta} O(Y) \ten \g \xla{\delta} O(Y) \ten \Lambda^2\g \la \ldots). 
\]
This carries a natural $G$-action (i.e.\ $O(G)$-coaction), and hence an action of the Lie algebra $\g$, and then we have a stacky CDGA $O([\oR\mu^{-1}(0)/\g])$, a double complex given by taking the Chevalley--Eilenberg cochain complex of $\g$ with coefficients in the chain complex $O(\oR\mu^{-1}(0))$. Hamiltonian structures on $\mu$ then give rise to $0$-shifted symplectic structures on $[\oR\mu^{-1}(0)/G]$ and $[\oR\mu^{-1}(0)/\g]$ by derived Lagrangian intersection.
\end{example}

\begin{remark}\label{cfCPTVV} 
Stacky CDGAs also arise as a step in the formulation of Poisson structures in 
\cite{CPTVV},  where they are called ``graded mixed cdgas'' --- see \cite[Remark 3.32]{poisson} for some more details.
\end{remark}

\begin{definition}
 Say that a morphism $U \to V$ of chain cochain complexes is a levelwise quasi-isomorphism if the map $U^i \to V^i$ of chain complexes is a quasi-isomorphism for all $i$. Say that a morphism of stacky CDGAs is a levelwise quasi-isomorphism if the underlying morphism of chain cochain complexes is so.
\end{definition}

The following adapts \cite[Lemma \ref{poisson-bicdgamodel}]{poisson}
\begin{lemma}\label{bicdgamodel}
There is a cofibrantly generated model structure on $DG^+dg_+\CAlg(R)$ in which fibrations are surjective in strictly positive chain degrees  and weak equivalences are levelwise quasi-isomorphisms. 
\end{lemma}

The following appear as \cite[Definitions 3.7 and 3.8]{poisson}. The functor $\widehat{\Tot}$ corresponds to the Tate realisation of \cite{CPTVV}.
\begin{definition}
 Given a chain cochain complex $V$, define the cochain complex $\widehat{\Tot} V \subset \Tot^{\Pi}V$ by
\[
(\widehat{\Tot} V)^m := (\bigoplus_{i < 0} V^i_{i-m}) \oplus (\prod_{i\ge 0}   V^i_{i-m})
\]
with differential $\pd \pm \delta$.
\end{definition}

\begin{definition}
 Given $A$-modules $M,N$ in chain cochain complexes, we define  internal $\Hom$ spaces
$\cHom_A(M,N)$  by
\[
 \cHom_A(M,N)^i_j=  \Hom_{A^{\#}_{\#}}(M^{\#}_{\#},N^{\#[i]}_{\#[j]}),
\]
with differentials  $\pd f:= \pd_N \circ f \pm f \circ \pd_M$ and  $\delta f:= \delta_N \circ f \pm f \circ \delta_M$,
where $V^{\#}_{\#}$ denotes the bigraded vector space underlying a chain cochain complex $V$. 

We then define the  $\Hom$ complex $\widehat{\HHom}_A(M,N)$ by
\[
 \widehat{\HHom}_A(M,N):= \widehat{\Tot} \cHom_A(M,N).
\]
\end{definition}

The reason for working with the functors $\widehat{\Tot}$ and $\widehat{\HHom}$ is that they send levelwise quasi-isomorphisms to quasi-isomorphisms. They also behave well with respect to tensor products.

\subsubsection{\'Etale morphisms}\label{detsn}

The definitions of \ref{etsn} now adapt to stacky CDGAs:

\begin{definition}
 We say that a morphism $f \co C' \to C$ in $DG^+dg_+\CAlg(R)$ is a square-zero extension if $f$ is surjective in every level and the kernel $I$ is a square-zero ideal. 
 
We say that $f$ is a contractible square-zero extension if in addition $I$ admits a $C'$-linear (or, equivalently, $C$-linear) contracting cochain homotopy compatible with the chain maps. Explicitly, this means that we have  maps $h \co I^n_j \to I^{n+1}_j$ for all $n,j$, satisfying $h\circ h=0$, $h\circ \pd +\pd\circ h =\id$ and $h\circ\delta +\delta\circ h=0$.  
\end{definition}

We will refer to functors as \emph{homotopy-preserving} if they preserve weak equivalences.

\begin{definition}\label{fetdef2}
Given homotopy-preserving functors $F,G$ from $dg_+DG^+\CAlg(R)$ to  simplicial sets, and a natural transformation $\eta \co F \to G$, we say that $\eta$ is homotopy formally  \'etale (resp. homotopy formally geometric)
if the maps
\[
 F(C') \to F(C)\by^h_{G(C)}G(C')
\]
 are $\pi_0$-surjective for all square-zero extensions (resp. all contractible square-zero extensions)
$C' \to C$. 

We say that a functor $F$ on $dg_+DG^+\CAlg(R)$ is homotopy  formally  \'etale  (resp. homotopy formally geometric)
if the transformation $F \to *$ to the constant functor is so. 
 \end{definition}

\begin{definition}\label{etdef2}
We say that  a natural transformation $\eta \co F \to G$ of functors $F,G$ on  $dg_+DG^+\CAlg(R)$ is l.f.p. if for any filtered colimit $C = \LLim_{i \in I} C_i$ in $DG^+\CAlg(R)$, the natural map
\[
 \LLim_{i \in I}F(C_i) \to \LLim_{i \in I}G(C_i)\by_{G(C)}F(C)
\]
is a weak  equivalence. 

We then say that $\eta$ is homotopy \'etale (resp. homotopy geometric)
if it is  homotopy formally  \'etale (resp. homotopy formally geometric)
and l.f.p. 

For $A \in dg_+DG^+\CAlg(R)$, we refer to $\Spec A$ as a stacky affine if the functor $\oR\Spec A := \map_{dg_+DG^+\CAlg(R)}(A,-)$ is geometric.
\end{definition}

Since any stacky CDGA $A \in dg_+DG^+\CAlg(R)$ gives rise to a simplicial set-valued functor $\oR \Spec $, the definitions above give rise to notions of homotopy (formally) 
\'etale maps between objects of $dg_+DG^+\CAlg(R)$. 

Standard obstruction arguments applied to obstruction classes in $\H_{-1}\z^0\cHom_B(\Omega^1_{B/A},M)$ lead to the following lemma, establishing a comparison  with the homotopy formally \'etale morphisms of \cite[\S \ref{poisson-bidescentsn}]{poisson}. 
\begin{lemma}\label{etlemma2}
 A cofibration $f \co A \to B$ is  homotopy formally geometric if and only if
\begin{itemize}
 \item for all $i,n>0$, we have $\H_i((\Omega^1_{B/A}\ten_B\H_0B^0)^n)=0$, and 
\item for all $n>0$, the $\H_0B^0$-module   $\H_0((\Omega^1_{B/A}\ten_B\H_0B^0)^n)$ is projective.
\end{itemize}

The morphism $f$ is  homotopy formally \'etale if and only if it is homotopy  formally geometric and
\begin{itemize}
 \item the inverse system  $\{\Tot \sigma^{\le q} (\Omega_{B/A}^1\ten_{B}\H_0B^0)\}_q$ of complexes of projective $\H_0B^0$-modules is pro-quasi-isomorphic to $0$,  where $\sigma$ denotes brutal truncation in the cochain direction.
\end{itemize}
\end{lemma}
Note that when $f$ is l.f.p., the final condition reduces to saying that the complexes  $\{\Tot \sigma^{\le q} (\Omega_{B/A}^1\ten_{B}\H_0B^0)\}_q$  of projective $\H_0B^0$-modules are acyclic for $q \gg 0$, because the inverse system stabilises.
 
\subsubsection{\'Etale atlases}
 
The denormalisation functor $D$ extends naturally to give a functor from stacky CDGAs to cosimplicial chain algebras.
Given a functor $F$ from chain algebras $dg_+\CAlg(R)$ to simplicial sets, we may adapt Definition \ref{Dlowerdef} to define a functor $D_*F$ on  $DG^+dg_+\CAlg(R)$ as the homotopy limit
\[
 D_*F(B):= \ho\Lim_{n \in \Delta} F(D^nB).
\]

This leads to the following generalisation of Theorem \ref{mainthm} to the derived setting. This has the same proof as Theorem \ref{mainthm} (with the proof of Lemma \ref{sharplemma} also generalising),  and  gives similar consequences for quasi-coherent complexes.
\begin{theorem}\label{mainthm2}
 Given an  $n$-geometric derived  Artin stack $\fX$ over $R$, the functor $D_*\fX$ is a hypersheaf which is $(n, \et)$-geometric with respect to the \'etale morphisms of Definition \ref{etdef2}.
\end{theorem}

Corollary \ref{gooddescent} also generalises with the same proof: the natural morphism
$
 \map(\fX,F) \to  \map(D_*\fX  , (D_*F)^{\sharp})
$
is an equivalence for all simplicial hypersheaves $F$ on $dg_+\CAlg(R)$.

\begin{example}\label{atlasYoverG2}
Example \ref{atlasYoverG} adapts to give stacky affine resolutions of quotient derived stacks $[Y/G]$, such as the derived Hamiltonian $[\oR\mu^{-1}(0)/G]$ reduction of Example \ref{Hamex}. Again, the  
 simplicial resolution is the nerve of the groupoid $([(Y \by G)/(\g \oplus \g)] \Rightarrow  [Y/\g])$ of stacky derived  affines.
\end{example}

\subsubsection{Tangent complexes and Poisson structures}

Analogues of the expressions in \S \ref{tgtsn} hold for stacky CDGAs, with the same references, so that in particular
\[
 \oR\HHom_{\sO_{\fX}}(\bL_{\fX}^{\ten p},\sO_{\fX}) \simeq  \oR\Gamma((D_*\fX)_{\et}, \widehat{\HHom}_{\sO}((\Omega^1_{\sO})^{\ten p},\sO)),
\]
for $(D_*\fX)_{\et}$ the site of \'etale maps $\Spec A \to (D_*\fX)_{\et}$ ($\infty$-localised at levelwise quasi-isomorphisms) for cofibrant geometric stacky CDGAs $A$.

This immediately leads to a formulation of shifted Poisson structures \cite[\S \ref{poisson-bipoisssn}]{poisson}. An $n$-shifted Poisson structure on a stacky derived affine $\Spec B$ is an enrichment of the CDGA structure on $\widehat{\Tot} B$ to a strong homotopy $P_{n+1}$-algebra structure, but with some boundedness restrictions on the (higher) Poisson brackets, which are required to lie in $\widehat{\HHom}_{B}((\Omega^1_B)^{\ten p},B)$.
An $n$-shifted Poisson structure on a derived Artin stack $\fX$ is then an  $\infty$-functorial choice of $n$-shifted Poisson structure for each stacky derived affine \'etale over $D_*\fX$.

\subsubsection{Deformation quantisation}\label{derivedquantsn}

All of the quantisations described in \S \ref{poissonquantsn} extend to the derived setting, once the relevant Hochschild complexes are defined by using $\widehat{\HHom}$ in the appropriate places. There are also several quantisation results which are only really interesting or new in derived or singular cases, which we now describe. 

Existence of curved $A_{\infty}$ quantisations of  $0$-shifted Poisson structures as almost commutative or $BD_1$-algebras,  formulated in terms of the Hochschild complex and allowing curvature, is given in \cite{DQnonneg,DQpoisson}.

\begin{example}\label{TstarBGmex}
   Consider the derived cotangent stack $T^*B\bG_m$ of the stack $B\bG_m$ (where $\bG_m=\GL_1$), which classifies line bundles.
 Like any cotangent stack, it carries a $0$-shifted symplectic structure, with associated Poisson structures on all stacky derived affines which are \'etale over it. This Poisson structure has a canonical deformation quantisation, giving rise to a deformation of the dg category of perfect complexes, which is described explicitly in \cite[Example 3.15]{DerCotMain}. 
 
 The dg category has a set of generators indexed by $\Z$, with no non-zero morphism between different generators. Endomorphisms of the generator of weight $n$ are given by the $R\llb \hbar \rrb$-CDGA $(R\llb \hbar \rrb \oplus R \llb \hbar \rrb t, \delta)$, with $t$ of cochain degree $-1$ and $\delta t = n\hbar$. On inverting $\hbar$,   these endomorphism complexes become acyclic for $n\ne 0$; that recovers the dg category of $R(\!( \hbar )\!)$-linear $\cD$-modules on $B\bG_m$,  so our quantised dg category (pre-inversion) has a much richer structure.
 
\end{example}

Existence of $(-1)$-shifted quantisations, as $BD_0$-algebras formulated in terms of differential operators, are given in \cite{DQvanish}.  For $(-2)$-shifted structures, the Beilinson--Drinfeld hierarchy of operads breaks down, but formulations are sometimes possible in terms of solutions of a quantum master equation \cite{DQ-2}. 

\subsection{Differentiable and analytic stacks}\label{Fermatsn} 

We now explain how to extend all of the results and formulations described so far  to differentiable and analytic stacks.

\subsubsection{Differentiable stacks}

In the differentiable setting, the analogue of a smooth geometric stacky affine is just an NQ-manifold. In other words, we have a CDGA
\[
 A^0 \xra{\pd} A^1 \xra{\pd} A^2\xra{\pd} \ldots 
 \]
where $A^0$ is the ring of $\cC^{\infty}$ functions on a some differentiable manifold $X_0$ and $\pd \co A^0 \to A^1$ is a $\cC^{\infty}$-derivation, with $A^{\#}$ freely generated over $A^0$ by a graded projective module (equivalently, functions on a graded vector bundle over $X_0$).

The analogue of a stacky affine in this setting drops all the freeness conditions, so $A^{\bt}$ is just an $\R$-CDGA equipped with a $\cC^{\infty}$-ring structure \cite{dubuc,MoerdijkReyes} on $A^0$ and the requirement that $\pd \co A^0 \to A^1$ be a $\cC^{\infty}$-derivation. $\cC^{\infty}$-rings can model singular spaces, often arising as quotients of rings of $\cC^{\infty}$ functions on  differentiable manifolds by arbitrary ideals. 

Making use of the theory of $\cC^{\infty}$-rings   and their homotopy theory as in \cite{CarchediRoytenbergHomological},
all of the results of \S\S \ref{stackyaffsn}--\ref{consequencesn} adapt to differentiable stacks. 
In this setting, \'etale and smooth morphisms correspond to local diffeomorphisms and submersions, respectively.
Most of the constructions are described in \cite{DQDG}, but in brief, every differentiable $n$-stack admits an \'etale $n$-atlas of NQ-manifolds. These atlases allow Poisson structures and deformation quantisations to be formulated in terms of  $\cC^{\infty}$-multiderivations and differential operators, and the algebraic existence and classification proofs adapt verbatim.

Introducing a second grading and a chain differential leads to derived structures as in  \cite{nuitenThesis}, with the results of \S \ref{derivedsn} then adapting to derived differentiable stacks --- again, see \cite{DQDG}.

\subsubsection{Analytic stacks}

In the complex analytic setting, the analogue of a smooth geometric stacky affine is a complex Stein NQ-manifold. In other words, we have a CDGA
\[
 A^0 \xra{\pd} A^1 \xra{\pd} A^2\xra{\pd} \ldots 
 \]
where $A^0$ is the ring of holomorphic functions on a some Stein manifold $X_0$ and $\pd \co A^0 \to A^1$ is an analytic derivation, with $A^{\#}$ freely generated over $A^0$ by a graded projective module (equivalently, functions on a graded vector bundle over $X_0$).

The analogue of a stacky affine in this setting drops all the freeness conditions, so $A^{\bt}$ is just a $\Cx$-CDGA equipped with a Stein algebra   structure on $A^0$ and the requirement that $\pd \co A^0 \to A^1$ be an analytic derivation. 

Again, introducing a second grading and a chain differential leads to derived structures, which are formulated in detail in \cite{DStein}, using the theory of rings with entire functional calculus. All of the results from the differentiable setting of \cite{DQDG} adapt to the analytic setting, and indeed to any Fermat theory.

In non-Archimedean settings, \cite{DStein} also gives a formulation of derived analytic geometry based on enriched Stein manifolds and Stein spaces, to which the results of \cite{DQDG} all adapt. The setup is based on overconvergent functions, and pro-objects crop up more than in other geometries, since we have to regard open polydiscs as inverse limits of the closed polydiscs containing them.   

\bibliographystyle{alphanum}
\addcontentsline{toc}{section}{Bibliography}
\bibliography{references}

\newcommand{\etalchar}[1]{$^{#1}$}
 \newcommand{\noop}[1]{} \def\cprime{$'$}
\begin{thebibliography}{{Sim}1}

\bibitem[Art]{Artin}
M.~Artin.
\newblock Versal deformations and algebraic stacks.
\newblock {\em Invent. Math.}, 27:165--189, 1974.

\bibitem[BPS]{DerCotMain}
Marco Benini, Jonathan~P. Pridham, and Alexander Schenkel.
\newblock Quantization of derived cotangent stacks and gauge theory on directed
  graphs.
\newblock {\em Adv. Theor. Math. Phys.}, to appear.
\newblock arXiv:2201.10225 [math-ph].

\bibitem[Cal]{calaqueLagrangian}
Damien Calaque.
\newblock Lagrangian structures on mapping stacks and semi-classical {TFT}s.
\newblock In {\em Stacks and categories in geometry, topology, and algebra},
  volume 643 of {\em Contemp. Math.}, pages 1--23. Amer. Math. Soc.,
  Providence, RI, 2015.

\bibitem[CPT{\etalchar{+}}]{CPTVV}
D.~Calaque, T.~Pantev, B.~To{\"e}n, M.~Vaqui{\'e}, and G.~Vezzosi.
\newblock Shifted {P}oisson structures and deformation quantization.
\newblock {\em J. Topol.}, 10(2):483--584, 2017\noop{2015}.
\newblock arXiv:1506.03699v4 [math.AG].

\bibitem[CR]{CarchediRoytenbergHomological}
D.~{Carchedi} and D.~{Roytenberg}.
\newblock {Homological Algebra for Superalgebras of Differentiable Functions}.
\newblock arXiv:1212.3745 [math.AG], 2012.

\bibitem[Dub]{dubuc}
Eduardo~J. Dubuc.
\newblock {$C^{\infty }$}-schemes.
\newblock {\em Amer. J. Math.}, 103(4):683--690, 1981.

\bibitem[GJ]{sht}
Paul~G. Goerss and John~F. Jardine.
\newblock {\em Simplicial homotopy theory}, volume 174 of {\em Progress in
  Mathematics}.
\newblock Birkh{\"a}user Verlag, Basel, 1999.

\bibitem[KV]{KhudaverdianVoronov}
H.~M. Khudaverdian and Th.~Th. Voronov.
\newblock Higher {P}oisson brackets and differential forms.
\newblock In {\em Geometric methods in physics}, volume 1079 of {\em AIP Conf.
  Proc.}, pages 203--215. Amer. Inst. Phys., Melville, NY, 2008.
\newblock arXiv:0808.3406v2 [math-ph].

\bibitem[LMB]{champs}
G{\'e}rard Laumon and Laurent Moret-Bailly.
\newblock {\em Champs alg{\'e}briques}, volume~39 of {\em Ergebnisse der
  Mathematik und ihrer Grenzgebiete. 3. Folge. A Series of Modern Surveys in
  Mathematics [Results in Mathematics and Related Areas. 3rd Series. A Series
  of Modern Surveys in Mathematics]}.
\newblock Springer-Verlag, Berlin, 2000.

\bibitem[Lur]{lurie}
J.~Lurie.
\newblock {\em Derived Algebraic Geometry}.
\newblock PhD thesis, M.I.T., 2004.
\newblock
  \href{http://hdl.handle.net/1721.1/30144}{http://hdl.handle.net/1721.1/30144}
  or
  \href{https://www.math.ias.edu/~lurie/papers/DAG.pdf}{https://www.math.ias.edu/$\sim$lurie/papers/DAG.pdf}.

\bibitem[MR]{MoerdijkReyes}
Ieke Moerdijk and Gonzalo~E. Reyes.
\newblock {\em Models for smooth infinitesimal analysis}.
\newblock Springer-Verlag, New York, 1991.

\bibitem[MS1]{MelaniSafronovI}
Valerio Melani and Pavel Safronov.
\newblock Derived coisotropic structures {I}: affine case.
\newblock {\em Selecta Math. (N.S.)}, 24(4):3061--3118, 2018.
\newblock arXiv:1608.01482 [math.AG].

\bibitem[MS2]{MelaniSafronovII}
Valerio Melani and Pavel Safronov.
\newblock Derived coisotropic structures {II}: stacks and quantization.
\newblock {\em Selecta Math. (N.S.)}, 24(4):3119--3173, 2018.
\newblock arXiv:1704.03201 [math.AG].

\bibitem[Nui]{nuitenThesis}
Joost Nuiten.
\newblock {\em Lie algebroids in derived differential topology}.
\newblock PhD thesis, Utrecht, 2018.

\bibitem[Pri1]{ddt1}
J.~P. Pridham.
\newblock Unifying derived deformation theories.
\newblock {\em Adv. Math.}, 224(3):772--826, 2010\noop{2007}.
\newblock corrigendum 228 (2011), no. 4, 2554--2556, arXiv:0705.0344v6
  [math.AG].

\bibitem[Pri2]{stacksintro}
J.~P. Pridham.
\newblock Notes characterising higher and derived stacks concretely.
\newblock arXiv:1105.4853v3 [math.AG], 2011.

\bibitem[Pri3]{stacks2}
J.~P. Pridham.
\newblock Presenting higher stacks as simplicial schemes.
\newblock {\em Adv. Math.}, 238:184--245, 2013\noop{2009}.
\newblock arXiv:0905.4044v4 [math.AG].

\bibitem[Pri4]{poisson}
J.~P. Pridham.
\newblock Shifted {P}oisson and symplectic structures on derived {$N$}-stacks.
\newblock {\em J. Topol.}, 10(1):178--210, 2017\noop{2015}.
\newblock arXiv:1504.01940v5 [math.AG].

\bibitem[Pri5]{DQ-2}
J.~P. Pridham.
\newblock Deformation quantisation for {$(-2)$}-shifted symplectic structures.
\newblock arXiv: 1809.11028v2 [math.AG], 2018.

\bibitem[Pri6]{DQnonneg}
J.~P. Pridham.
\newblock Deformation quantisation for unshifted symplectic structures on
  derived {A}rtin stacks.
\newblock {\em Selecta Math. (N.S.)}, 24(4):3027--3059, 2018.
\newblock arXiv: 1604.04458v4 [math.AG].

\bibitem[Pri7]{DQDG}
J.~P. Pridham.
\newblock An outline of shifted {P}oisson structures and deformation
  quantisation in derived differential geometry.
\newblock arXiv: 1804.07622v3 [math.DG], 2018.

\bibitem[Pri8]{DQvanish}
J.~P. Pridham.
\newblock Deformation quantisation for {$(-1)$}-shifted symplectic structures
  and vanishing cycles.
\newblock {\em Algebr. Geom.}, 6(6):747--779, 2019.
\newblock arXiv:1508.07936v5 [math.AG].

\bibitem[Pri9]{DQpoisson}
J.~P. Pridham.
\newblock Quantisation of derived {P}oisson structures.
\newblock arXiv: 1708.00496v5 [math.AG], 2019.

\bibitem[Pri10]{DStein}
J.~P. Pridham.
\newblock A differential graded model for derived analytic geometry.
\newblock {\em Advances in Mathematics}, 360:106922, 2020\noop{2018}.
\newblock arXiv: 1805.08538v1 [math.AG].

\bibitem[Pri11]{DQLag}
J.~P. Pridham.
\newblock Quantisation of derived {L}agrangians.
\newblock {\em Geom. Topol.}, to appear.
\newblock arXiv: 1607.01000v4 [math.AG].

\bibitem[PTVV]{PTVV}
T.~Pantev, B.~To{\"e}n, M.~Vaqui{\'e}, and G.~Vezzosi.
\newblock Shifted symplectic structures.
\newblock {\em Publ. Math. Inst. Hautes \'Etudes Sci.}, 117:271--328,
  2013\noop{2011}.
\newblock arXiv: 1111.3209v4 [math.AG].

\bibitem[{Sim}1]{simpsonAlgGeomNstacks}
C.~{Simpson}.
\newblock {Algebraic (geometric) $n$-stacks}.
\newblock arXiv:alg-geom/9609014, 1996.

\bibitem[Sim2]{simpsonHtpy}
Carlos Simpson.
\newblock Homotopy over the complex numbers and generalized de {R}ham
  cohomology.
\newblock In {\em Moduli of vector bundles ({S}anda, 1994; {K}yoto, 1994)},
  volume 179 of {\em Lecture Notes in Pure and Appl. Math.}, pages 229--263.
  Dekker, New York, 1996.

\bibitem[TV]{hag2}
Bertrand To{\"e}n and Gabriele Vezzosi.
\newblock Homotopical algebraic geometry. {II}. {G}eometric stacks and
  applications.
\newblock {\em Mem. Amer. Math. Soc.}, 193(902):x+224, 2008\noop{2004}.
\newblock arXiv math.AG/0404373 v7.

\bibitem[Wei]{W}
Charles~A. Weibel.
\newblock {\em An introduction to homological algebra}.
\newblock Cambridge University Press, Cambridge, 1994.

\end{thebibliography}
\end{document}